\documentclass[psamsfonts]{amsart}

\usepackage{amsmath, bbm}
\usepackage{amssymb}
\usepackage{amsthm}
\usepackage[margin=1in]{geometry}
\usepackage{fancyhdr}
\usepackage{todonotes}
\usepackage{setspace}
\usepackage[all,arc]{xy}
\usepackage{enumerate}
\usepackage{mathrsfs}

\usepackage{color}
\usepackage[colorlinks=true,urlcolor=blue,citecolor=red,linkcolor=blue,linktocpage,pdfpagelabels,bookmarksnumbered,bookmarksopen]{hyperref}

\newtheorem{thm}{Theorem}[section]
\newtheorem{cor}[thm]{Corollary}
\newtheorem{prop}[thm]{Proposition}
\newtheorem{lem}[thm]{Lemma}

\theoremstyle{definition}
\newtheorem{defn}[thm]{Definition}

\theoremstyle{remark}
\newtheorem{rem}[thm]{Remark}

\newtheorem{claim}{Claim}

\DeclareMathOperator{\spt}{spt}

\DeclareMathOperator{\id}{id}

\DeclareMathOperator{\diam}{diam}

\DeclareMathOperator{\dist}{dist}

\DeclareMathOperator{\Vol}{Vol}
\DeclareMathOperator{\Rad}{Rad}

\newcommand{\g}{\gamma}
\newcommand{\Om}{\Omega}

\newcommand{\cK}{\mathcal{K}}
\newcommand{\cL}{\mathcal{L}}
\newcommand{\cP}{\mathcal{P}}

\newcommand{\R}{\mathbf{R}}

\newcommand{\mres}{\mathbin{\vrule height 1.6ex depth 0pt width
		0.13ex\vrule height 0.13ex depth 0pt width 1.3ex}}

\newcommand{\paren}[1]{\left(#1\right)}
\newcommand{\bracket}[1]{\left[#1\right]}

\renewcommand{\d}{\,\textnormal{d}}

\numberwithin{equation}{section}

\title{On nonexpansiveness of metric projection operators on Wasserstein spaces}

\author[A. Adve]{Anshul Adve} 
\address{Department of Mathematics, Princeton University, Princeton, NJ 08540, USA}
\email{aadve@princeton.edu}

\author[A.R. M\'esz\'aros]{Alp\'ar R. M\'esz\'aros}  
\date{\today}
\address{Department of Mathematical Sciences, University of Durham, Durham DH1 3LE, United Kingdom}
\email{alpar.r.meszaros@durham.ac.uk}

\date{\today}

\begin{document}
	
	\maketitle
	
	\begin{abstract}
		In this paper we investigate properties of metric projections onto specific closed and geodesically convex proper subsets of Wasserstein spaces $(\cP_p(\R^d),W_p).$ When $d=1$, as $(\cP_2(\R),W_2)$ is isometrically isomorphic to a flat space with a Hilbertian structure, the corresponding projection operators are expected to be nonexpansive. We give a direct proof of this fact, relying on intrinsic analysis, which also implies nonexpansiveness in certain special cases in higher dimensions. When $d>1$, we show the failure of this property in two regimes: when $p>1$ is either small enough or large enough. Finally, we prove some positive curvature properties of Wasserstein spaces $(\cP_p(\R^d),W_p)$ when $d\ge 2$ and $p\in(1,+\infty)$ are arbitrary: we show that Wasserstein spaces are nowhere locally {\it Busemann NPC} spaces, and they nowhere locally satisfy the so-called \emph{projection criterion}. As a corollary of the former, they have nonnegative \emph{upper Alexandrov curvature}, in a precise sense that we define here. In our analysis a particular subset of probability measures having densities uniformly bounded above by a given constant plays a special role.
	\end{abstract}

	
	\section{Introduction}
	
	Fix a closed, convex set $\Omega \subseteq \mathbf{R}^d$. For $p \ge 1$, let $\mathcal{P}_p(\Omega)$ denote the set of nonnegative Borel probability measures $\mu$ supported in $\Omega$ with finite $p^{th}$ moment $\int_\Omega |x|^p \d\mu < \infty$. Equip this space with the $p$-Wasserstein distance $W_p$, i.e.
	$$
	W_p^p(\mu,\nu):=\inf\left\{\int_{\Om\times\Om}|x-y|^p\d\g(x,y):\ \g\in\Pi(\mu,\nu) \right\},
	$$
	where $\Pi(\mu,\nu):=\left\{\g\in\cP_p(\Om\times\Om):\ (\pi^x)_\sharp\g=\mu,\ (\pi^y)_\sharp\g=\nu\right\}$ denotes the set of transportation plans between $\mu$ and $\nu$ and $\pi^x,\pi^y:\Om\times\Om\to\Om$ stand for the canonical projections $\pi^x(a,b)=a$, $\pi^y(a,b)=b$. We denote by $\Pi_o(\mu,\nu)\subseteq\Pi(\mu,\nu)$ the set of optimal plans that realize the value $W_p(\mu,\nu)$. It is well-known (see for instance \cite{AmbGigSav}) that $(\cP_p(\Om),W_p)$ defines a geodesic metric space. Moreover, if $\Om$ is compact then $W_p$ metrizes the weak-$\star$ convergence of probability measures in $\cP_p(\Om).$ For convenience, we sometimes use the notation $W_p(\Omega):=(\cP_p(\Om),W_p)$.
	
	In this paper, we are interested in properties of projection operators ${\rm{P}}_{\cK}^p:\cP_p(\Om)\to \cK$, where $\cK\subseteq\cP_p(\Om)$ is a given closed and geodesically convex proper subset of $\cP_p(\Om)$. In particular, the main question we are interested in is the so-called {\it nonexpansiveness property} that reads as
	\begin{align}\label{question:nonexpansive?}
		{\rm Is\ it\ true\ that\ \ }W_p\left({\rm{P}}_{\cK}^p[\mu], {\rm{P}}_\cK^p[\nu]\right)\le W_p(\mu,\nu),\ \forall\ \mu,\nu\in\cP_p(\Om)\ ?\tag{Q}
	\end{align}
	A closely related question, which we also address, is to what extent such nonexpansiveness properties (or their failure) reveal new features of Wasserstein spaces, from the point of view of their curvature.
	
	For $\mu\in\cP_p(\Om)$, the projection ${\rm{P}}_\cK^p[\mu]$  is defined as the solution of the variational problem
	\begin{equation}\label{def:proj}
		{\rm{P}}_{\cK}^p[\mu]:={\rm{argmin}}\left\{W_p^p(\rho,\mu):\ \rho\in\cK\right\}.
	\end{equation}

	A few comments on the definition of this operator are necessary. The existence of a solution in this minimization problem is an easy consequence of the direct method of calculus of variations. Indeed, for $\mu\in\cP_p(\Om)$ and $C>0$, the set $\{\rho\in\cP_p(\Om): \ W_p(\rho,\mu)\le C\}$ is tight and the objective functional is weakly lower semicontinuous with respect to the narrow convergence of probability measures. However, for ${\rm{P}}_{\cK}^p[\mu]$ to be well-defined, we would need to have the uniqueness of a minimizer in \eqref{def:proj}. This turns out to be a subtle question and it is linked to the strict convexity of $\rho\mapsto W_p^p(\rho,\mu)$ and/or the curvature properties of $(\cP_p(\Om),W_p)$. 
	
	While $\rho\mapsto W_p^p(\rho,\mu)$ is known to be convex with respect to the `flat' convex combination of probability measures, i.e. along $[0,1]\ni t\mapsto (1-t)\rho_0+ t\rho_1$, its strict convexity typically fails, unless additional conditions are imposed on $\mu$ (for instance absolute continuity with respect to $\cL^d\mres\Om$; see \cite[Proposition 7.17-7.19]{San}). From the geometric viewpoint however, when studying properties of projection operators, it is more natural to consider the notion of geodesic convexity (which is also referred to as {\it displacement convexity} in the case of $(\cP_p(\Om),W_p)$; see \cite{McC, AmbGigSav}). This notion is intimately linked to the curvature properties of the space. By \cite[Section 7.3]{AmbGigSav} we know that when $d\ge 2$, $(\cP_2(\Om),W_2)$ is a positively curved space in the sense of Alexandrov, and so the mapping $\rho\mapsto W_2^2(\rho,\mu)$ in general is not geodesically  $\lambda$-convex, for any $\lambda\in\R$.
	Similarly, the failure of the uniqueness of geodesics in general in $(\cP_p(\Om),W_p)$ indicates that all these spaces are non-negatively curved also for $p\neq 2$ (cf. \cite[Corollary 2.3.2]{Jos}). However, to the best of our knowledge, the literature on fine curvature properties of $(\cP_p(\Om),W_p)$ is very sparse at this point.
	
	These considerations let us conclude that the uniqueness of the projection onto closed and geodesically convex sets $\cK$ fails in general. To illustrate this fact, consider the following example. Let $\Om=\R^2$, and let
	\begin{align*}
		\cK:=\left\{\frac12\delta_{(-x,2)}+\frac12\delta_{(x,-2)}:\ x\in[-1,1]\right\}.
	\end{align*}
	Then $\cK$ is a closed geodesically convex set in $(\cP_2(\Om),W_2)$. Let $\mu:=\frac12\delta_{(-1,0)}+\frac12\delta_{(1,0)}$. Clearly, both measures $\rho_0:=\frac12\delta_{(-1,2)}+\frac12\delta_{(1,-2)}$ and $\rho_1:=\frac12\delta_{(1,2)}+\frac12\delta_{(-1,-2)}$ belong to $\cK$ and have the same minimal $W_2$ distance from $\mu$. So the projection of $\mu$ onto $\cK$ cannot be defined in a unique way in this case.
	
	Because of this reason, in part of our study we will focus on some particular geodesically convex closed subsets $\cK\subset\cP_p(\Om)$ onto which we can guarantee the uniqueness of the projected measure in \eqref{def:proj}.
	
	For a given $\lambda>0$, we consider 
	\begin{align}\label{def:Kp}
		\cK^\lambda_p(\Omega) := \left\{\rho \in \mathcal{P}_p(\Omega) \cap L^1(\Omega) \text{ and } 0 \leq \rho \leq \lambda {\rm{\ a.e.}}\right\},
	\end{align}
	that is the subset of absolutely continuous probability measures having densities uniformly bounded above by $\lambda$. We use the notation $\cK_p(\Omega)$ for $\cK^1_p(\Omega)$.
	Some of our main results concern the properties of the projection operator ${\rm{P}}_{\cK_{p}(\Omega)}^p$. To simplify this (when it will be clear from the context), we use the notation 
		\begin{equation}\label{simp:notation}
			{\rm{P}}_{\Omega}^p := {\rm{P}}_{\cK_{p}(\Omega)}^p.
		\end{equation}

	As we show in Lemma \ref{lem:admissibles_convex}, $\cK_p(\Om)$ is a closed geodesically convex subset of $(\cP_p(\Om),W_p)$. More importantly, arguments verbatim to the ones in \cite[Proposition 5.2]{DePMesSanVel} (which give a \emph{saturation} characterization of the projected measure) let us conclude that the projection problem \eqref{def:proj} onto $\cK_p(\Om)$ has a unique solution for any $p>1$ (only the case $p=2$ was considered in \cite{DePMesSanVel}).
	A secondary motivation behind the consideration of these particular subsets is the following: in recent years the set $\cK_2(\Om)$ received some special attention in applications of optimal transport techniques to study the well-posedness and further properties of PDEs arising in crowd motion models under {\it density constraints}. For a non-exhaustive list of references on this subject we refer to \cite{MauRouSan,MesSan, DePMesSanVel,Mes}.  
	
	\medskip
	
	\noindent {\bf A non-exhaustive literature review of the curvature properties of $p$-Wasserstein spaces.} The study of curvature properties of metric spaces has a long and fruitful history. For a very good exposition of the subject we refer to the monographs \cite{BriHae,Jos}. Compared to this, the study of the curvature properties of Wasserstein spaces started only relatively recently. It seems that \cite[Section 4.5]{Ott} gave the first formal arguments for the fact that $W_2(\R^d)$ has nonnegative (sectional) curvature, and that $W_2(\mathbf{R}^d)$ is flat for $d=1$ and non-flat for $d>1$. A similar claim has been made in \cite[Corollary 2]{Lott} (in the case when the ambient space is a smooth Riemannian manifold with nonnegative sectional curvature). Positive curvature of the 2-Wasserstein space in the sense of Alexandrov has been discussed in \cite[Proposition 2.10]{Stu:1} and \cite[Section 7.3]{AmbGigSav} (for metric spaces with positive curvature in the sense of Alexandrov, as ambient spaces) and in \cite[Theorem A.8]{LotVil} (in the case of smooth Riemannian manifolds with nonnegative sectional curvature, as ambient spaces). The vanishing curvature of $W_2(\R)$ has been mentioned also in \cite{Klo}. On the other hand, \cite[Remark 2.10]{BerKlo} explains that non-positive curvature of a metric space $X$ is not generally inherited by $W_2(X)$, because $W_2(X)$ can be non-geodesic.
	It seems to us, however, that a detailed study of the curvature properties of $p$-Wasserstein spaces, for $p\neq 2$, is not available in the literature. Therefore, in this paper we propose some potential steps in this direction.
	
	\medskip \medskip
	
	\noindent{\bf Description of main results.} Rather than investigating general curvature properties of Wasserstein spaces, our main focus in this paper is to obtain some fine properties of metric projection operators. In turn, these properties reveal some features of these spaces from the point of view of their curvature, in a quantitative way.
	
	\medskip
	
	{\it Nonexpansiveness of ${\rm{P}}_\Om^p$ onto  $\cK_p(\Om)$.}
	
	\medskip
	
	First, we investigate the question of nonexpansiveness of the projection operator ${\rm{P}}_\Om^p$ onto the set $\cK_p(\Om)$. When $d=1$, it is well-known that $\cP_2(\R)$ is isometrically isomorphic to a closed convex subset of a Hilbert space (the space of nondecreasing functions belonging to $L^2([0,1];\R)$, see \cite[Section 9.1]{AmbGigSav}). The isomorphism takes a probability measure on $\R$ to the inverse of its cumulative distribution function. Therefore, $(\cP_2(\R),W_2)$ can be regarded as a flat space and so it is expected that ${\rm{P}}_\R^2$ is nonexpansive onto closed geodesically convex subsets $\cK\subset\cP_2(\R)$. Indeed, every closed geodesically convex subset $\cK\subset\cP_2(\R)$ corresponds to a closed convex subset of $L^2([0,1];\R)$. For instance, the space $\cK_2(\R)$ defined in \eqref{def:Kp} corresponds to $\{X\in L^2([0,1];\R):\ X'\ge 1\ {\rm{a.e.}}\}$. Therefore, the projection problem from $(\cP_2(\R),W_2)$ onto $\cK_2(\R)$ can be transferred to a projection problem in a Hilbertian setting, which has the nonexpansive property. Returning to the original setting via the isometric isomorphism, it follows that $\textnormal{P}_{\mathbf{R}}^2$ is nonexpansive. When $p\neq 2$, the nonexpansiveness property of projection operators on $L^p$ spaces is a more subtle question (see for instance \cite{Bru}) and therefore a conclusion similar to the one when $p=2$ seems to be nontrivial. In the case when $d>1$, even for $p=2$, it is not possible to identify $(\cP_2(\R^d),W_2)$ with a subset of a Hilbert space (and in particular, as discussed before, this space will not be flat). Therefore in those cases, the question of nonexpansiveness seems to be highly nontrivial.
	
	When $p=2$, Theorem \ref{thm:almost_lipschitz} presents a sort of {\it weak nonexpansiveness property} of the projection in arbitrary dimensions. Here we show that the left hand side of the inequality in \eqref{question:nonexpansive?} is always bounded above by the transportation cost of a certain suboptimal plan between the original measures. This suboptimal plan becomes optimal in two extreme scenarios: either when $d=1$ or when one of the original measures $\mu,\nu$ is a Dirac mass. So, this yields the nonexpansiveness of the projection operator when $d=1$ (see Corollary \ref{cor:1d}) or when one of the measures is a Dirac mass (see Corollary \ref{cor:dirac_nonexpansive}). 
	
	By \cite[Theorem 2.3]{KriRep} and \cite[Proposition 2.4 of Chapter II.2]{BriHae}, we know that both smooth Riemannian manifolds with non-positive sectional curvature and Alexandrov spaces with non-positive curvature satisfy the projection nonexpansiveness property. To the best of our knowledge, it is unclear whether the non-positive curvature condition of these spaces in general (beyond Riemannian manifolds) is also a necessary condition to ensure the nonexpansiveness of the projection operator in general. 
	
	When $d\ge 2$, as previously discussed, $(\cP_p(\Om),W_p)$ possesses some non-negative curvature properties (even for $p\neq 2$). Therefore, there is good reason to anticipate that there exist closed geodesically convex subsets $\mathcal{K} \subseteq \cP_p(\Om)$ onto which the projection operator ${\rm{P}}_{\mathcal{K}}^{p}$ fails to be nonexpansive (whenever it is well-defined). Indeed, we show in the second half of Section \ref{sec:results} that in the case $\mathcal{K} = \mathcal{K}_p(\Omega)$, nonexpansiveness fails in two regimes: either when $p>1$ is small enough (Proposition \ref{prop:small_p_counterexample}) or when it is large enough (Proposition \ref{prop:p_large}).

	More specifically, Proposition \ref{prop:small_p_counterexample} shows that there exists $p(d)>1$ small such that for any $p\in (1,p(d))$, the projection operator ${\rm{P}}_{\R^d}^p$ onto $\cK_p(\R^d)$ fails to be nonexpansive. In our construction, we provide a quantitative asymptotic description of $p(d)$ as a function of $d$, for $d$ large. Proposition \ref{prop:p_large} shows that there exists a closed convex set $\Omega \subseteq \mathbf{R}^d$ and a universal constant $C \geq 1$ (independent of the dimension) such that ${\rm P}_{\Omega}^p$ fails to be nonexpansive for all $p \geq Cd$. 
	
	The proofs of both of these propositions are constructive, i.e. we construct particular counterexamples to the nonexpansiveness property. Interestingly, relying again on Corollary \ref{cor:dirac_nonexpansive}, neither of the previous constructions provide a counterexample for $p=2$. Heuristically, our results would provide an argument  (in combination with \cite[Proposition 2.4 of Chapter II.2]{BriHae}) for the fact that $(\cP_p(\Om),W_p)$ is positively curved, when $p>1$ is specified in the two regimes given by these propositions.
	
	\medskip
	
	{\it Some nonnegative curvature properties of Wasserstein spaces.}
	
	\medskip
	
	Based on the results in hand regarding the nonexpansiveness properties of the projection operators, we go on and study further curvature properties of Wasserstein spaces $(\cP_p(\R^d),W_p)$ in a non-heuristic sense, without restrictions on $p$. In metric geometry, it is well-known that the so-called {\it Busemann non-positively curved (NPC)} spaces admit a natural characterization via a projection nonexpansivity property. In Section \ref{sec:5} we show that Wasserstein spaces fail to be Busemann NPC spaces. It is also well-known that every Alexandrov NPC space is a Busemann NPC space. Therefore, Wasserstein spaces also fail to be Alexandrov NPC spaces.
	Even though the failure of the NPC properties both in the Busemann and Alexandrov sense can be seen as  a consequence of non-uniqueness of some geodesics, we demonstrate this failure in a robust sense even in situations where geodesics are unique.

Our main results from Section \ref{sec:5} are summarized in Theorem \ref{thm:Wp_not_busemann} and Theorem \ref{thm:Wp_not_nonexpansive}. Here, we show in particular that if $d \geq 2$ and $1 \leq p \leq \infty$, no open subspace of $W_p(\mathbf{R}^d)$ is a Busemann NPC space. This implies that $W_p(\mathbf{R}^d)$ has positive upper curvature in the sense of Definition \ref{defn:upper_curvature}.

\medskip

For simplicity of presentation of our main ideas, in this manuscript we consider only the case when the ambient space is a (subset of a) flat Euclidean space $\R^d$. We expect our results to remain true in the case of smooth Riemannian manifolds that satisfy suitable nonnegative curvature bounds. However, we leave these questions to future study.

\medskip

The structure of the rest of the paper is as follows. In Section \ref{sec:3} we study some geometric properties of the projection operator ${\rm{P}}_\Om^2$. These properties seem to be interesting in their own right: we show that ${\rm{P}}_{\R^d}^2$ preserves barycenters of measures (see Proposition \ref{prop:barycenters}), and it satisfies a certain translation invariance with respect to distances between measures (see Proposition \ref{prop:translation_invariance}).  Section \ref{sec:results} contains the proofs of our first main results: in Theorem \ref{thm:almost_lipschitz} we show the `weak nonexpansiveness' property of ${\rm{P}}_\Om^2$, and deduce the full nonexpansiveness in the two cases mentioned above (see Corollaries \ref{cor:1d} and \ref{cor:dirac_nonexpansive}). Furthermore, Proposition \ref{prop:small_p_counterexample} constructs the counterexample to the nonexpansiveness of ${\rm{P}}_{\R^d}^p$ onto $\cK_p(\R^d)$ when $d\ge 2$ and $p\in (1,p(d))$, and studies the asymptotic behavior of $p(d)$ as the dimension becomes large. Here we give also the proof of Proposition \ref{prop:p_large}, i.e. the failure of nonexpansiveness when $p$ is large enough. Finally, in Section \ref{sec:5}, we study the nonnegative curvature properties of Wasserstein spaces from the point of view of Busemann NPC spaces, for general $p\in (1,+\infty)$. This section also contains the proofs of our last two main results, Theorems \ref{thm:Wp_not_busemann} and \ref{thm:Wp_not_nonexpansive}. We end the paper with Appendix \ref{sec:prelim}, where we recall some results from optimal transport used in the main text.

\section{Some geometric properties of ${\rm{P}}_\Om^2$}\label{sec:3}

\subsection{Barycenters and translation invariance of ${\rm{P}}_\Om^2$}

Recall the definition of $\cK_{p}(\Omega)$ from \eqref{def:Kp} and the simplified notation ${\rm P}_{\Omega}^{p}$ from \eqref{simp:notation}. Suppose for now that $\Omega = \mathbf{R}^d$, so we do not have to worry about boundaries. Then there are a few symmetries which one can exploit in Question \ref{question:nonexpansive?}. First, the projection operator ${\rm{P}}_{\Omega}^p$ commutes with translations. When $p=2$, the projection also preserves barycenters:

\begin{prop}\label{prop:barycenters}
	Let $\mu \in \mathcal{P}_2( \mathbf{R}^d)$ and $\rho:= {\rm{P}}_{ \mathbf{R}^d}^2[\mu]$. Then
	\begin{align*}
		\int_ {\mathbf{R}^d} x \d\rho = \int_{ \mathbf{R}^d} x \d\mu.
	\end{align*}
\end{prop}

\begin{proof}
	First, let us note that by Lemma \ref{lem:P_defined}, the projection $\rm{P}_{\mathbf{R}^d}^2$ is well-defined.
	For $h \in \mathbf{R}^d$, let $\tau \colon x \mapsto x+h$ denote the translation map by $h$. Then $\tau_{\sharp}\rho \in \cK_2( \mathbf{R}^d)$. Let $\gamma$ be an optimal plan between $\mu$ and $\rho$. Then by Lemma \ref{lem:plan_translate} $(\id,\tau)_{\sharp}\gamma$ is optimal for $W_2(\mu,\tau_{\sharp}\rho)$. Thus, by the optimality of both $\rho$ and $\gamma$,
	\begin{align*}
		W^2_2(\mu,\rho)&=\int_{ \mathbf{R}^d\times \mathbf{R}^d} |x-y|^2 \d\gamma\\
		& \leq W_2^2(\mu,\tau_{\sharp}\rho)= \int_{ \mathbf{R}^d\times\mathbf{R}^d} |x-y|^2 \d[(\id,\tau)_{\sharp}\gamma]\\
		& = \int_{ \mathbf{R}^d\times \mathbf{R}^d} |x-y-h|^2 \d\gamma = \int_{ \mathbf{R}^d\times \mathbf{R}^d}|x-y|^2\d\gamma - 2h \cdot \int_{ \mathbf{R}^d\times  \mathbf{R}^d} (x-y)\d\gamma + |h|^2.
	\end{align*}
	We conclude that $\displaystyle\int_{ \mathbf{R}^d\times \mathbf{R}^d}(x-y)\d\gamma = 0$. Indeed, otherwise one could set $\displaystyle h:=\lambda\int_{ \mathbf{R}^d\times \mathbf{R}^d}(x-y)\d\gamma$ for $\lambda>0$, and would obtain
		\begin{align*}
			W^2_2(\mu,\rho)&\le W^2_2(\mu,\rho) - 2\lambda \left|\int_{ \mathbf{R}^d\times  \mathbf{R}^d} (x-y)\d\gamma\right|^{2} + \lambda^{2} \left|\int_{ \mathbf{R}^d\times  \mathbf{R}^d} (x-y)\d\gamma\right|^{2}\\
			&= W^2_2(\mu,\rho) + \left(\lambda^{2}-2\lambda\right) \left|\int_{ \mathbf{R}^d\times  \mathbf{R}^d} (x-y)\d\gamma\right|^{2}.
		\end{align*}
		So, by choosing $\lambda\in (0,2)$, the previous inequality would yield a contradiction. The result follows, since
		\begin{align*}
			0 &= \displaystyle\int_{ \mathbf{R}^d\times \mathbf{R}^d}(x-y)\d\gamma = \int_{ \mathbf{R}^d} x \d\mu - \int_{ \mathbf{R}^d} x \d\rho.
			\qedhere
		\end{align*}
\end{proof}

\begin{lem}\label{lem:plan_translate}
	Let $\mu,\nu\in\mathcal{P}_2(\mathbf{R}^d)$ and let $\nu'\in\mathcal{P}_2(\mathbf{R}^d)$ a translation of $\nu$, i.e. $\nu'=\tau_\sharp\nu$, where $\tau:x\mapsto x+h$ (for some given $h\in\mathbf{R}^d$). If $\gamma\in\mathcal{P}_2(\mathbf{R}^d\times\mathbf{R}^d)$ is optimal for $W_2^2(\mu,\nu)$, then $(\id,\tau)_\sharp\gamma$ is optimal for $W_2(\mu,\nu').$ 
\end{lem}

\begin{proof}
	It is immediate to check that $\tilde\gamma:=(\id,\tau)_\sharp\gamma$ is an admissible plan for $W_2(\mu,\nu').$
	
	By \cite[Theorem 5.10]{Vil} (see also \cite[Section 1.6.2]{San}) it is enough to show that $\tilde\gamma$ has cyclic monotone support. Let $n\in\mathbb{N}$. We notice that a collection of $n$ points from $\spt(\tilde\gamma)$ has the form $(x_i,y_i+h)_{i=1}^n$, where $(x_i,y_i)\in\spt(\gamma)$, $i\in\{1,\dots,n\}$. Let $\sigma:\{1,\dots,n\}\to\{1,\dots,n\}$ be a permutation of $n$ letters. Then we have
	\begin{align*}
		\sum_{i=1}^n |x_i-y_i-h|^2 & = \sum_{i=1}^n |x_i-y_i|^2-2\sum_{i=1}^n(x_i-y_i)\cdot h + n|h|^2\\
		&\le \sum_{i=1}^n |x_i-y_{\sigma(i)}|^2-2\sum_{i=1}^n(x_i-y_i)\cdot h + n|h|^2 = \sum_{i=1}^n |x_i-y_{\sigma(i)}-h|^2,
	\end{align*}
	where in the inequality we have used the cyclic monotonicity of $\spt(\gamma)$. The result follows.
\end{proof}

From these observations one obtains the following ``translation invariance" when $p = 2$ and the ambient space is $ \mathbf{R}^d$.

\begin{prop}\label{prop:translation_invariance}
	Let $\mu,\nu \in \mathcal{P}_2(\mathbf{R}^d)$ and $\nu'$ a translate of $\nu$. Then
	\begin{align*}
		W_2^2(\mu,\nu) - W_2^2(\textnormal{P}_{ \mathbf{R}^d}^2[\mu],\textnormal{P}_{ \mathbf{R}^d}^2[\nu]) = W_2^2(\mu,\nu') - W_2^2(\textnormal{P}_{ \mathbf{R}^d}^2[\mu],\textnormal{P}_{ \mathbf{R}^d}^2[\nu']).
	\end{align*}
\end{prop}

\begin{proof}
	Denote $\rho := \mathcal{P}_{ \mathbf{R}^d}^2[\mu]$, $\sigma := \mathcal{P}_{ \mathbf{R}^d}^2[\nu]$, and $\sigma' := \mathcal{P}_{ \mathbf{R}^d}^2[\nu']$. Let $\gamma\in\Pi_o(\mu,\nu)$ and $\eta\in\Pi_o(\rho,\sigma)$. If $\tau \colon x \mapsto x+h$ is the translation map which pushes forward $\nu$ onto $\nu'$, then we can construct optimal plans $\gamma',\eta'$ from $\mu,\rho$ to $\nu',\sigma'$, respectively, by $\gamma' = (\id,\tau)_{\sharp}\gamma$ and $\eta' = (\id,\tau)_{\sharp}\eta$ (see Lemma \ref{lem:plan_translate}; here we have also used the fact that the projection of the translate of a measure is the translate of the projection). Thus
	\begin{align*}
		W_2^2(\mu,\nu') - W_2^2(\rho,\sigma')
		&= \int_{ \mathbf{R}^d\times \mathbf{R}^d} |x-y-h|^2 \d\gamma - \int_{ \mathbf{R}^d\times \mathbf{R}^d} |x-y-h|^2 \d\eta
		\\&= \int_{ \mathbf{R}^d\times  \mathbf{R}^d} |x-y|^2 \d\gamma - \int_{ \mathbf{R}^d\times \mathbf{R}^d} |x-y|^2 \d\eta\\
		& - 2h \cdot \int_{ \mathbf{R}^d\times \mathbf{R}^d} (x-y) \d\gamma + 2h \cdot \int_{ \mathbf{R}^d\times \mathbf{R}^d} (x-y) \d\eta
		\\&= W_2^2(\mu,\nu) - W_2^2(\rho,\sigma) + 2h \cdot \paren{\int_{ \mathbf{R}^d} x \d\rho - \int_{ \mathbf{R}^d} x \d\mu} + 2h \cdot \paren{\int_{ \mathbf{R}^d} y \d\nu - \int_{ \mathbf{R}^d} y \d\sigma},
	\end{align*}
	and the last two terms vanish by Proposition \ref{prop:barycenters}.
\end{proof}

In particular, any counterexample $\mu,\nu$ to nonexpansiveness must remain a counterexample when $\mu,\nu$ are replaced by translates of themselves. This already eliminates several candidates $\mu,\nu$ that may seem like potential counterexamples at first sight.

\section{Nonexpansiveness vs. failure of nonexpansiveness for ${\rm P}_{\Omega}^p$}\label{sec:results}

Recall the definition of $\cK_{p}(\Omega)$ from \eqref{def:Kp} and the simplified notation ${\rm P}_{\Omega}^{p}$ from \eqref{simp:notation}.
Throughout this section, let $\mu,\nu \in \mathcal{P}_p(\Omega)$, and set $\tilde\mu := {\rm{P}}_{\Omega}^p[\mu]$ and $\tilde\nu := {\rm{P}}_{\Omega}^p[\nu]$. Denote the optimal transport plan from $\tilde\mu$ to $\tilde\nu$ by $\eta$. Note that since $\tilde\mu,\tilde\nu$ are absolutely continuous (see Lemma \ref{lem:P_defined}), $\eta$ is induced by a map. 

\subsection{Weak nonexpansiveness of the projection when $p = 2$.}

In the theorem below one bounds the distance squared between $\mu$ and $\nu$ by the transportation cost of a slightly suboptimal transport plan. This is a sort of ``weak nonexpansiveness.''

\begin{thm}\label{thm:almost_lipschitz}
	Let $\Omega\subseteq \mathbf{R}^d$ be a closed convex set. Let $T,U:\Omega\to\Omega$ stand for the  optimal maps from $\tilde\mu,\tilde\nu$ to $\mu,\nu$, respectively. Take $p = 2$ and $\gamma := (T,U)_{\sharp}\eta\in\Pi(\mu,\nu)$. Then
	\begin{align*}
		W_2^2(\tilde\mu,\tilde\nu) \leq \int_{\Omega\times\Omega} |x-y|^2 \d\gamma(x,y).
	\end{align*}
\end{thm}

\begin{proof}
	One can write
	\begin{align*}
		\int_{\Omega\times\Omega} |x-y|^2 \d\gamma
		&= \int_{\Omega\times\Omega} |T(x)-U(y)|^2 \d\eta
		= \int_{\Omega\times\Omega} |x-y + T(x)-x + y-U(y)|^2 \d\eta
		\\&= \int_{\Omega\times\Omega} |x-y|^2 \d\eta + 2\int_{\Omega\times\Omega} (x-y) \cdot (T(x)-x+y-U(y)) \d\eta\\
		& + \int_{\Omega\times\Omega} |T(x)-x+y-U(y)|^2 \d\eta
		\\&\geq \int_{\Omega\times\Omega} |x-y|^2 \d\eta + 2\int_{\Omega\times\Omega} (x-y) \cdot (T(x)-x) \d\eta + 2\int_{\Omega\times\Omega} (y-x) \cdot (U(y)-y) \d\eta.
	\end{align*}
	Thus it suffices to show that
	\begin{align*}
		\int_{\Omega\times\Omega} (x-y) \cdot (T(x)-x) \d\eta \geq 0
		\qquad \text{and (by symmetry)} \qquad
		\int_{\Omega\times\Omega} (y-x) \cdot (U(y)-y) \d\eta \geq 0.
	\end{align*}
	For $t \in (0,1)$, let $\pi_t(x,y) := (1-t)x+ty$. Then $\tilde\mu_t:=(\pi_t)_{\sharp}\eta \in \cK_2(\Omega)$ by the geodesic convexity of $\mathcal{K}_2(\Omega)$. The optimality of $\tilde\mu$ in the definition of $\textnormal{P}_{\Omega}^2[\mu]$, together with the fact that $\tilde\eta:=(T,\pi_t)_\sharp\eta\in \Pi(\mu,\tilde\mu_t)$, implies
	\begin{align*}
		W_2^2(\mu,\tilde\mu)
		&\le W_2^2(\mu,\tilde\mu_t)\leq \int_{\Omega\times\Omega}|x-y|^2\d\tilde\eta= \int_{\Omega\times\Omega} |T(x)-\pi_t(x,y)|^2 \d\eta
		= \int_{\Omega\times\Omega} |T(x)-x + t(x-y)|^2 \d\eta\\
		&= \int_{\Omega\times\Omega} |T(x)-x|^2 \d\eta + 2t \int_{\Omega\times\Omega} (x-y) \cdot (T(x)-x) \d\eta + t^2 \int_{\Omega\times\Omega} |x-y|^2 \d\eta\\
		&=W_2^2(\mu,\tilde\mu)+ 2t \int_{\Omega\times\Omega} (x-y) \cdot (T(x)-x) \d\eta+ t^2W_2^2(\tilde\mu,\tilde\nu).
	\end{align*}
	Thus, we have obtained 
	\begin{align*}
		-tW_2^2(\tilde\mu,\tilde\nu)\le 2 \int_{\Omega\times\Omega} (x-y) \cdot (T(x)-x) \d\eta.
	\end{align*}
	
	Letting $t \rightarrow 0$, we conclude that
	\begin{align*}
		\int_{\Omega\times\Omega} (x-y) \cdot (T(x)-x) \d\eta \geq 0,
	\end{align*}
	as desired.
\end{proof}

When $\Omega \subseteq \mathbf{R}$ is an interval, this theorem is enough to deduce that the answer to Question \ref{question:nonexpansive?} is \emph{yes}.

\begin{cor}\label{cor:1d}
	Suppose that $\Omega \subseteq \mathbf{R}$ is an interval. Then ${\rm{P}}_{\Omega}^2$ is nonexapansive.
\end{cor}

\begin{proof}
	The plan $\gamma$ defined in Theorem \ref{thm:almost_lipschitz} is monotonically increasing, hence optimal.
\end{proof}

Theorem \ref{thm:almost_lipschitz} also implies nonexpansiveness when $\Pi(\mu,\nu)$ is a singleton. This is the case if and only if one of the measures $\mu,\nu$ is a Dirac mass.

\begin{cor}\label{cor:dirac_nonexpansive}
	Suppose that $\mu,\nu$ are such that $\Pi(\mu,\nu)$ is a singleton. Then
	\begin{align*}
		W_2(\tilde\mu,\tilde\nu) \leq W_2(\mu,\nu).
	\end{align*}
\end{cor}

\begin{proof}
	There is only one transport plan between $\mu$ and $\nu$, so using the notation of Theorem \ref{thm:almost_lipschitz}, $\gamma\in\Pi(\mu,\nu)$ must be this plan. The result follows.
\end{proof}

\begin{rem}
	In general, $\gamma\in\Pi(\mu,\nu)$ in the statement of Theorem \ref{thm:almost_lipschitz} does not need to be optimal. In the case of $\Omega = \mathbf{R}^2$, consider for instance $\mu = \frac{1}{2}\delta_{(R,0)} + \frac{1}{2}\delta_{(-R,0)}$ and $\nu = \frac{1}{2}\delta_{(t,1)} + \frac{1}{2}\delta_{(-t,-1)}$, where $R$ is large and $t$ is small. For $t > 0$, the optimal map from $\mu$ to $\nu$ sends all the mass from $(R,0)$ to $(t,1)$, and all the mass from $(-R,0)$ to $(-t,-1)$. On the other hand, for $t < 0$, the optimal map sends all the mass from $(R,0)$ to $(-t,-1)$, and all the mass from $(-R,0)$ to $(t,1)$. This means that the optimal plan from $\mu$ to $\nu$ does not vary continuously with $t$ (it is discontinuous at $t = 0$). However, one can see that the plan $\gamma$ does depend continuously on $t$, so it cannot be optimal.
\end{rem}

\subsection{Failure of nonexpansiveness of ${\rm P}_{\R^d}^p$ for $d>1$ and $p$ small}

For $p$ very close to $1$, the proof of the following proposition illustrates a counterexample to the nonexpansive property of ${\rm{P}}_{\mathbf{R}^d}^p$ onto $\cK_p(\mathbf{R}^d)$. Although our argument requires $p$ to be near 1, it remains an open question whether a restriction on $p$ is necessary.

\begin{prop}\label{prop:small_p_counterexample}
	Let $\Omega = \mathbf{R}^d$ with $d > 1$. Then there exists $p(d) > 1$ such that  ${\rm{P}}_{\Omega}^p$ fails to be nonexpansive with respect to the $p$-Wasserstein distance for $1 < p < p(d)$. In fact, one can take
	\begin{align*}
		p(d) = 1 + \frac{1}{O(d^2 \log d)}.
	\end{align*}
\end{prop}
In the proof below we use the following conventions: for positive quantities $A,B$ possibly depending on various parameters, we write $A \lesssim B$ if $A \leq CB$ with $C$ an absolute constant, $A \gtrsim B$ if $B \lesssim A$, and $A \sim B$ if $A \lesssim B \lesssim A$.

\begin{proof}[Proof of Proposition \ref{prop:small_p_counterexample}]
	Let $R>0$ be the radius of the ball of volume $\frac{1}{2}$. Let $\mu = \frac{1}{2}\delta_{(0,\dots,0)} + \frac{1}{2}\delta_{(2R,0,\dots,0)}$ and $\nu = \delta_{(0,\dots,0)}$. Then $\tilde\mu = {\rm{P}}_{\Omega}^p[\mu]$ and $\tilde\nu = {\rm{P}}_{\Omega}^p[\nu]$ are the restriction of Lebesgue measure to
	\begin{align*}
		\{x \in \mathbf{R}^d : |x| \leq R \text{ or } |x-(2R,0,\dots,0)| \leq R\}
		\qquad \text{and} \qquad
		\{x \in \mathbf{R}^d : |x| \leq 2^{1/d}R\}
	\end{align*}
	respectively. Let $\gamma := \mu\otimes\nu$, which is the only element in $\Pi(\mu,\nu)$, and let $\eta\in\Pi(\tilde\mu,\tilde\nu)$ be an arbitrary transport plan.
	We will show that
	\begin{align}\label{eqn:W1_ctrexample}
		\int_{\mathbf{R}^d\times\mathbf{R}^d} |x-y| \d\eta > \int_{\mathbf{R}^d\times\mathbf{R}^d} |x-y| \d\gamma
	\end{align}
	with an explicit lower bound on the difference; then we obtain the desired inequality
	\begin{align*}
		\int_{\mathbf{R}^d\times\mathbf{R}^d} |x-y|^p \d\eta > \int_{\mathbf{R}^d\times\mathbf{R}^d} |x-y|^p \d\gamma
	\end{align*}
	for $p \in [1,p(d))$ by continuity in $p$.
	
	The right hand side of \eqref{eqn:W1_ctrexample} is necessarily equal to
	\begin{align}\label{eqn:W1_mu_nu}
		\int_{\mathbf{R}^d\times\mathbf{R}^d}|x-y|\d\gamma = R = \int_{\R^d} x_1 \d\tilde\mu - \int_{\R^d} y_1 \d\tilde\nu = \int_{\mathbf{R}^d\times\mathbf{R}^d} (x_1-y_1) \d\eta.
	\end{align}
	Thus to get a quantitative form of \eqref{eqn:W1_ctrexample}, it is enough to estimate
	\begin{align}\label{eqn:to_lbd}
		\int_{\mathbf{R}^d\times\mathbf{R}^d} (|x-y|-|x_1-y_1|) \d\eta
	\end{align}
	from below. Given $x \in \mathbf{R}^d$, denote $x' = (x_2,\dots,x_d) \in \mathbf{R}^{d-1}$. Let
	\begin{align*}
		E = \{y \in \mathbf{R}^d \colon 1.1^{1/d}R \leq |y'| \leq 1.9^{1/d}R \text{ and } |y_1| \leq \sqrt{2^{2/d}-1.9^{2/d}}R\} \subseteq \{x \in \mathbf{R}^d \colon |x| \leq 2^{1/d}R\} = \spt\tilde\nu.
	\end{align*}
	Suppose $(x,y) \in \spt\eta$ with $y \in E$. Then $|x'| \leq R$, so
	\begin{align*}
		|x'-y'| \geq (1.1^{1/d}-1)R
		= R \int_{0}^{1/d} 1.1^t \log 1.1 \d t
		\gtrsim R/d,
	\end{align*}
	On the other hand
	\begin{align*}
		|x-y| \leq \diam\{\spt\tilde\mu \cup \spt\tilde\nu\} \lesssim R.
	\end{align*}
	Combining these two facts yields
	\begin{align*}
		|x-y| - |x_1-y_1| \gtrsim \Big(\sqrt{1+\frac{1}{d^2}} - 1\Big)R \gtrsim \frac{R}{d^2}.
	\end{align*}
	Plugging this into \eqref{eqn:to_lbd} and recalling \eqref{eqn:W1_mu_nu}, we deduce that
	\begin{align}\label{eqn:abstract_lbd}
		\int_{\R^d\times\R^d} |x-y| \d\eta - \int_{\R^d\times\R^d} |x-y| \d\gamma &\geq \int_{\R^d\times\R^d} (|x-y| - |x_1-y_1|) \d\eta\\
		&\nonumber \geq \int_{\R^d \times E} (|x-y|-|x_1-y_1|) \d\eta
		\gtrsim \tilde\nu(E) \frac{R}{d^2}.
	\end{align}
	In computing $\tilde\nu(E)$, it will be convenient to write $\Vol_n(r)$ for the volume of the $n$-dimensional ball of radius $r$, and $\Rad_n(v)$ for the radius of the $n$-dimensional ball of volume $v$. Then $R = \Rad_d(1/2)$ by definition, and by classical formulas, for any $n\in\mathbb{N}$
	\begin{align*}
		\Vol_n(r) = \frac{\pi^{n/2}}{\Gamma(n/2+1)} r^n
		\qquad \text{ and } \qquad
		\Rad_n(v) = \frac{\Gamma(n/2+1)^{1/n}}{\sqrt{\pi}} v^{1/n} \sim \sqrt{n} v^{1/n}.
	\end{align*}
	Thus $R \sim \sqrt{d}$, and
	\begin{align*}
		\tilde\nu(E) &= 2\sqrt{2^{2/d} - 1.9^{2/d}} R [\Vol_{d-1}(1.9^{1/d}R) - \Vol_{d-1}(1.1^{1/d}R)]
		\\&\sim \frac{R}{\sqrt{d}} \Vol_{d-1}(R)
		= \frac{1}{\sqrt{d}} \frac{R\Vol_{d-1}(R)}{\Vol_d(R)} \Vol_d(R)
		\sim \frac{1}{\sqrt{d}} \frac{\Gamma(d/2+1)}{\Gamma((d-1)/2+1)}
		\sim 1,
	\end{align*}
	where the final estimate follows from Stirling's asymptotic for the Gamma function and from the fact that by the choice of $R$, $\Vol_d(R)=\frac12$.
	From \eqref{eqn:abstract_lbd} we therefore conclude
	\begin{align}\label{eqn:W1_case}
		\int_{\R^d\times\R^d} |x-y| \d\eta - \int_{\R^d\times\R^d} |x-y| \d\gamma \gtrsim \frac{1}{d^{3/2}}.
	\end{align}
	
	The inequality we need to prove is
	\begin{align}\label{eqn:goal}
		\int_{\R^d\times\R^d} |x-y|^p \d\gamma < \int_{\R^d\times\R^d} |x-y|^p \d\eta
	\end{align}
	for $p \in (1,p(d))$. If $c > 0$ is the implied constant in \eqref{eqn:W1_case}, then \eqref{eqn:goal} will follow from \eqref{eqn:W1_case} as long as $p$ is small enough that
	\begin{align*}
		\bracket{\int_{\R^d\times\R^d}|x-y|^p\d\gamma - \int_{\R^d\times\R^d}|x-y|\d\gamma} + \bracket{\int_{\R^d\times\R^d}|x-y|\d\eta - \int_{\R^d\times\R^d}|x-y|^p\d\eta} < \frac{c}{d^{3/2}}.
	\end{align*}
	The first term in brackets is simply
	\begin{align*}
		\int_{\R^d\times\R^d} |x-y|^p \d\gamma - \int_{\R^d\times\R^d} |x-y| \d\gamma  = 2^{p-1}R^p - R.
	\end{align*}
	Because of the general inequality $t-t^p \leq p-1$ for all $t \geq 0$ and $p\ge 1$, the second term in brackets must be at most $p-1$. Thus \eqref{eqn:goal} holds whenever
	\begin{align*}
		2^{p-1}R^p - R + p - 1 < \frac{c}{d^{3/2}}.
	\end{align*}
	One can estimate
	\begin{align*}
		2^{p-1}R^p - R \lesssim (p-1) R^p \log R
		\lesssim (p-1) d^{p/2} \log d
	\end{align*}
	for $p$ bounded, so it is enough if
	\begin{align*}
		(p-1)(1+d^{p/2}\log d) < \frac{c'}{d^{3/2}}
	\end{align*}
	for some smaller absolute constant $c' > 0$. This is true for
	\begin{align*}
		p < 1 + \frac{1}{O(d^2 \log d)},
	\end{align*}
	as long as the implied constant is sufficiently large.
\end{proof}
\begin{rem}\label{rem:1_ctrexmp_failure}
	Let us note as a consequence of Corollary \ref{cor:dirac_nonexpansive} that this construction is not a counterexample when $p =2$.
\end{rem}

\subsection{Failure of nonexpansiveness of ${\rm P}_{\Omega}^p$ for $d>1$ and $p$ large}

\begin{prop}\label{prop:p_large}
	Let $d > 1$. Then there is a closed convex set $\Omega \subseteq \mathbf{R}^d$ and a universal constant $C \geq 1$ such that ${\rm P}_{\Omega}^p$ fails to be nonexpansive with respect to the $p$-Wasserstein distance for all $p \geq Cd$.
\end{prop}

\begin{proof}
	Take $\Omega$ to be the cone
	\begin{align*}
		\Omega := \{x = (x_1,x') \in \mathbf{R} \times \mathbf{R}^{d-1} = \mathbf{R}^d \colon |x'| \leq \varepsilon x_1\},
	\end{align*}
	where $\varepsilon \in (0,1]$ is a small constant to be chosen later. Let $R = \Rad_d(1)$ denote the radius of the $d$-dimensional ball of unit volume. Set $S = 3\varepsilon^{-1}R$, and let $\nu = \delta_{(S,0,\dots,0)}$ be a Dirac mass at the point $S$ along the $x_1$ axis. Then the ball of unit volume around $(S,0,\dots,0)$ is contained in $\Omega$, so ${\rm P}_{\Omega}^p[\nu]$ is the Lebesgue measure restricted to this ball.
	
	For $t \geq 0$, let
	\begin{align*}
		\Omega_t := \{x \in \Omega : x_1 \leq t\}.
	\end{align*}
	Denote the volume of the unit $(d-1)$-dimensional ball by $V = \Vol_{d-1}(1)$. Write $\mathcal{L}^d$ for the $d$-dimensional Lebesgue measure. Then
	\begin{align*}
		\mathcal{L}^d(\Omega_t)
		= \int_0^t \mathcal{L}^{d-1}(\{x' \in \mathbf{R}^{d-1} : |x'| \leq \varepsilon x_1\}) \d x_1
		= V \varepsilon^{d-1} \int_0^t x_1^{d-1} \d x_1
		= \frac{V}{d} \varepsilon^{d-1} t^d.
	\end{align*}
	Set
	\begin{align*}
		s = \frac{1}{2}\Big(\frac{d}{V}\Big)^{\frac{1}{d}} \varepsilon^{-(1-\frac{1}{d})},
	\end{align*}
	so that $\mathcal{L}^d(\Omega_{2s}) = 1$. Let $\mu = \delta_{(s,0,\dots,0)}$. By saturation \cite[Proposition 5.2]{DePMesSanVel}, ${\rm P}_{\Omega}^p [\mu]$ is Lebesgue measure on $\spt \left({\rm P}_{\Omega}^p[\mu]\right)$, which is the intersection of $\Omega$ and a ball centered at $(s,0,\dots,0)$. This ball must have radius at least $s$, as otherwise it would be strictly contained in $\Omega_{2s}$, which has volume $1$. In particular, $\spt\left( {\rm P}_{\Omega}^p[\mu]\right) \supset \Omega_{s/2}$. The distance between $\Omega_{s/2}$ and $\spt \left({\rm P}_{\Omega}^p[\nu]\right)$ is $S-R-s/2$. From this we deduce the lower bound
	\begin{align*}
		W_p^p({\rm{P}}_{\Omega}^p[\mu], {\rm{P}}_{\Omega}^p [\nu])
		\geq \mathcal{L}^d(\Omega_{s/2}) \Big(S-R-\frac{s}{2}\Big)^p
		= \frac{1}{4^d} \Big(S-R-\frac{s}{2}\Big)^p.
	\end{align*}
	On the other hand,
	\begin{align*}
		W_p^p(\mu,\nu) = (S-s)^p.
	\end{align*}
	Thus it suffices to show that we can choose $\varepsilon$ such that
	\begin{align*}
		\frac{1}{4^d}\Big(S-R-\frac{s}{2}\Big)^p
		> (S-s)^p
	\end{align*}
	whenever $p \geq Cd$ for some absolute constant $C$. If we pick $\varepsilon$ such that $s > 2R$, then this inequality is equivalent to
	\begin{align}\label{eqn:p_large}
		p > \Big[\log_4 \frac{S-R-s/2}{S-s}\Big]^{-1} d
		= \Big[\log_4 \Big(1 + \frac{s/2-R}{S-s}\Big)\Big]^{-1} d
	\end{align}
	Let us underline that the condition $s > 2R$ is necessary to ensure that this logarithm is positive. Furthermore, by construction we always have that $S>s$, for $\varepsilon$ small.
	By the standard formula for the volume of a ball and Stirling's approximation,
	\begin{align*}
		R \sim \sqrt{d}
		\qquad \text{and} \qquad
		V^{\frac{1}{d-1}} \sim \frac{1}{\sqrt{d-1}}.
	\end{align*}
	The latter is equivalent to $V^{\frac1d}\sim \frac{1}{\sqrt{d}}$. Therefore
	\begin{align*}
		s \sim \sqrt{d} \varepsilon^{-(1-\frac{1}{d})}
		\qquad \text{and} \qquad
		S \sim \sqrt{d} \varepsilon^{-1}.
	\end{align*}
	These asymptotics show that if $\varepsilon$ is a sufficiently small absolute constant, then the constraint $s > 2R$ will be met, and
	\begin{align*}
		\frac{s/2-R}{S-s} \sim \varepsilon^{\frac{1}{d}}.
	\end{align*}
	Thus the inequality \eqref{eqn:p_large} holds as long as $p \geq Cd$, where the constant $C$ can be taken to satisfy $C\lesssim \varepsilon^{-1/d} \lesssim 1$. Let us underline that $C$ can be chosen to be a universal constant, i.e., independent of $\Omega,d,\mu,\nu$. We saw that the above choices of $\Omega,\mu,\nu$ provide a counterexample to nonexpansivenss of ${\rm P}_{\Omega}^p$ whenever \eqref{eqn:p_large} holds, so the proof is complete.
\end{proof}

\begin{rem}\label{rem:infty_ctrexmp_failure}
	We note, as in Remark \ref{rem:1_ctrexmp_failure}, that by Corollary \ref{cor:dirac_nonexpansive} this cannot be a counterexample when $p=2$.
\end{rem}

\section{Nonnegative curvature properties of Wasserstein spaces}\label{sec:5}

In this section we show that Wasserstein spaces behave as if they are nowhere non-positively curved. There are several standard notions of non-positive curvature for metric spaces. In particular, below we recall the definition of a Busemann non-positively curved (NPC) space, and we state a criterion for non-positive curvature in terms of projections. It is well-known that a Riemannian manifold has non-positive sectional curvature if and only if it is a Busemann NPC space, and if and only if it satisfies the projection criterion (see \cite{BriHae} and \cite[Theorem 2.3]{KriRep}). We show that Wasserstein spaces dramatically fail both of these criteria. As a corollary, Wasserstein spaces have nonnegative upper curvature in the sense of Definition \ref{defn:upper_curvature}.

\begin{defn}[Busemann NPC space, \cite{Jos}]\label{def:busemann}
	A Busemann NPC space is a metric space which admits an open cover by geodesic subspaces, such that whenever $[0,1]\ni t\mapsto\gamma_t$ and $[0,1]\ni t\mapsto\eta_t$ are geodesics in one of these subspaces and $\gamma_0 = \eta_0$, the function $t\mapsto\dist(\gamma_t,\eta_t)$ is convex.
\end{defn}

\begin{thm}{\cite[Corollary 2.3.1]{Jos}}\label{thm:alex-buse}
	Every Alexandrov NPC space is a Busemann NPC space.
\end{thm}

\begin{defn}[Upper Alexandrov curvature] \label{defn:upper_curvature}
	Let $X$ be a locally geodesic metric space. We define the upper (Alexandrov) curvature at a point $x \in X$ to be
	\begin{align*}
		\kappa(x) := \inf_{(U(x),\kappa'(x))} \kappa'(x),
	\end{align*}
	where the infimum is over all pairs $(U(x),\kappa'(x))$ with $U(x)$ a geodesic neighborhood of $x$ of Alexandrov curvature $\leq \kappa'(x)$.
	If $X$ is a Riemannian manifold, this is the maximum of the sectional curvatures at $x$. We say that $X$ has nonnegative upper curvature if $\kappa(x) \geq 0$ for all $x \in X$.
\end{defn}

\begin{defn}
	Given a subset $S$ of a metric space and a point $x$ in the space, write
	\begin{align*}
		{\rm{P}}_S(x) := \{s \in S : \dist(x,s) = \dist(x,S)\}.
	\end{align*}
	This is the (set-valued) projection of $x$ onto $S$.
\end{defn}

\begin{defn}[Proximinal subset]
	Suppose $X$ is a metric space and $S \subseteq U \subseteq X$. Then $S$ is \emph{$U$-proximinal} (respectively, \emph{$U$-Chebyshev}) if ${\rm{P}}_S(x)$ is nonempty (respectively, a singleton) for all $x \in U$. We drop the prefix $U$- when $U = X$.
\end{defn}

\begin{defn}[Projection criterion]
	A metric space $X$ satisfies the projection criterion for non-positive curvature if it admits an open cover $X = \bigcup_\alpha U_{\alpha}$ by geodesic subspaces, such that whenever $S \subseteq U_{\alpha}$ is geodesically convex and $U_{\alpha}$-proximinal, one has
	\begin{align*}
		\dist({\rm{P}}_S(x),{\rm{P}}_S(y)) \leq \dist(x,y)
		\qquad\text{for all } x,y \in U_{\alpha}.
	\end{align*}
\end{defn}

\begin{rem}
	The content of the previous definition is \cite[Definition 2.2]{KriRep}, and we view this as a ``criterion for non-positive curvature" because of \cite[Theorem~2.3]{KriRep}.
\end{rem}

Our main results from this section are summarized in the following theorems.

\begin{thm}\label{thm:Wp_not_busemann}
	For $d \geq 2$ and $1 \leq p \leq \infty$, no open subspace of $W_p(\mathbf{R}^d)$ is a Busemann NPC space.
\end{thm}

\begin{thm}\label{thm:Wp_not_nonexpansive}
	For $d \geq 2$ and $1 < p < \infty$, no open subspace of $W_p(\mathbf{R}^d)$ satisfies the projection criterion.
\end{thm}

As a consequence of Theorem \ref{thm:Wp_not_busemann} and Theorem \ref{thm:alex-buse}, we have the following characterization.

\begin{cor}
	For $d \geq 2$ and $1 \leq p \leq \infty$, the Wasserstein space $W_p(\mathbf{R}^d)$ has positive upper curvature.
\end{cor}

To prove these theorems, we produce ``global counterexamples" to non-positive curvature in the lemmas below, and then we scale down these counterexamples to embed them into any open subspace of $W_p(\mathbf{R}^d)$. Similar arguments work with $W_p(\mathbf{R}^d)$ replaced by $W_p(\Omega)$ or $\mathcal{K}_p(\Omega)$.

\begin{lem}\label{lem:non_busemann}
	Let $d \geq 2$ and $1 \leq p \leq \infty$. Let $\varepsilon > 0$. Then there are geodesics $(\mu_t)_{t\in[0,1]},(\nu_t)_{t\in[0,1]}$ in $W_p(\mathbf{R}^d)$ with $\mu_0 = \nu_0$ such that $W_p(\mu_1,\nu_1) \lesssim \varepsilon$ but $W_p\left(\mu_{\frac{1}{2}},\nu_{\frac{1}{2}}\right) \gtrsim 1$.
\end{lem}

\begin{proof}
	We illustrate this in $\mathbf{R}^2$; the same construction works in $\mathbf{R}^d$ by embedding $\mathbf{R}^2$ in $\mathbf{R}^d$ in the usual way. Consider the geodesics
	\begin{align*}
		\mu_t = \frac{1}{2} \delta_{(-t,1+\varepsilon-t)} + \frac{1}{2} \delta_{(t,-1-\varepsilon+t)}
		\qquad &\text{and} \qquad
		\nu_t = \frac{1}{2} \delta_{(t,1+\varepsilon-t)} + \frac{1}{2} \delta_{(-t,-1-\varepsilon+t)}, \ t\in[0,1].
	\end{align*}
	These clearly satisfy the claims of the lemma.
\end{proof}

\begin{lem}\label{lem:projection_fail}
	Let $d \geq 2$ and $1 < p < \infty$. Let $\varepsilon > 0$. Then there are measures $\mu,\nu \in W_p(\mathbf{R}^d)$ and a geodesically convex and Chebyshev subset $\mathcal {S} \subseteq W_p(\mathbf{R}^d)$ such that $W_p(\mu,\nu) \lesssim \varepsilon$ but $W_p({\rm{P}}^{p}_\mathcal{S}[\mu], {\rm{P}}^{p}_\mathcal{S}[\nu]) \gtrsim 1$.
\end{lem}

\begin{proof}
	Again we illustrate this in $\mathbf{R}^2$; an almost identical construction works in $\mathbf{R}^d$. Let $Q^{\pm}$ be the rectangles
	\begin{align*}
		Q^+ = [-1,1] \times [2,3]
		\qquad \text{and} \qquad
		Q^- = [-1,1] \times [-3,-2],
	\end{align*}
	and set $Q = Q^+ \cup Q^-$. Fix $1 < p < \infty$. Take $\mathcal{S} \subseteq W_p(\mathbf{R}^2)$ to be the set of measures $\mu$ supported in $Q$ with density bounded above by $\varepsilon^{-10}$ and with $\mu(Q^+) = \mu(Q^-) = \frac{1}{2}$.

	\begin{claim} \label{claim:proximal}
		The subset $\mathcal{S} \subseteq W_p(\mathbf{R}^2)$ is Chebyshev.
	\end{claim}
	
	{\it Proof of Claim \ref{claim:proximal}.}
	
	The proof is along the lines of the saturation argument which proves \cite[Proposition~5.2]{DePMesSanVel} (take $f=\mathbbm{1}_Q$ in the statement of that proposition). The only difference between our setting and the setting of \cite{DePMesSanVel} is that we impose the additional constraint on measures in $S$ that they should give equal mass to $Q^+$ and $Q^-$. This forces us to work a little harder than in \cite{DePMesSanVel} (without this additional constraint, the construction of the set $K$ in the next three paragraphs could be significantly simplified), but the underlying idea is the same.
	
	Let $\mu \in W_p(\mathbf{R}^2)$. We must show that ${\rm{P}}_\mathcal{S}^{p}[\mu]$ is a singleton. This set is certainly nonempty by weak-$\star$ compactness and lower semi-continuity of $W_p(\cdot,\mu)$ with respect to the weak-$\star$ topology. Therefore it suffices to show that if $\rho_1,\rho_2 \in {\rm{P}}_\mathcal{S}^{p}[\mu]$, then $\rho_1 = \rho_2$. Suppose for a contradiction that $\mathcal{L}^2(\{\rho_1 \neq \rho_2\})>0$, where $\mathcal{L}^2$ denotes the Lebesgue measure on $\mathbf{R}^2$. Since $\rho_i$, $i = 1,2$, is absolutely continuous, and the cost function is strictly convex, there is a unique optimal transport plan $\eta_i$ from $\rho_i$ to $\mu$, and $\eta_i$ is induced by a map $T_i$, i.e. $\eta_i=(\id,T_i)_\sharp \rho_i$. This $T_i$ is defined and unique $\rho_i$-a.e. Equivalently, $T_i$ is defined and unique $\mathcal{L}^2$-a.e. on $\{\rho_i \neq 0\}$.
	
	Let $\rho = \frac{1}{2}\rho_1 + \frac{1}{2}\rho_2 \in \mathcal{S}$. Consider the plan $\eta = \frac{1}{2}\eta_1 + \frac{1}{2}\eta_2$ from $\rho$ to $\mu$. By construction, the cost of this plan is exactly $W_p(\rho_1,\mu) = W_p(\rho_2,\mu)$, so by optimality of $\rho_1,\rho_2$ for the projection ${\rm{P}}_\mathcal{S}^{p}$, we must have $\rho \in {\rm{P}}_\mathcal{S}^{p}[\mu]$, and $\eta$ must be an optimal plan from $\rho$ to $\mu$. It follows from the optimality of $\eta$ and the absolute continuity of $\rho$ that $\eta$ is induced by a map $T$, so $\eta$ is supported on the graph of $T$. But $\eta$ is nonzero on the union of the graphs of $T_1,T_2$, so the only way $\eta$ can be supported on a graph is if $T_1,T_2,T$ agree on $\{\rho_1 \neq 0\} \cap \{\rho_2 \neq 0\}$.
	
	Now consider $T_i|_{\{\rho_1 \neq \rho_2\}}$. It cannot happen that this is the identity map for both $i=1,2$, because then the left hand side of the following equation is independent of $i$, but the right hand side is not:
	\begin{align*}
		\mu
		= (T_i)_{\sharp} \rho_i
		&= (T_i)_{\sharp}(\rho_i \mres\{\rho_1 = \rho_2\} + \rho_i \mres \{\rho_1 \neq \rho_2\})
		\\&= T_{\sharp}(\rho_i \mres \{\rho_1=\rho_2\}) + \rho_i \mres \{\rho_1 \neq \rho_2\}
	\end{align*}
	(in the last equality here, the first terms agree because $T_1 = T_2 = T$ on $\{\rho_1 \neq 0\} \cap \{\rho_2 \neq 0\}$, and the second terms agree because $T_1=T_2=\id$ on $\{\rho_1 \neq \rho_2\}$).
	
	Thus $T|_{\{\rho_1 \neq \rho_2\}} \neq \id$. By the definition of $\mathcal{S}$, we have $\rho_1,\rho_2 \leq \varepsilon^{-10}$, so the strict inequality $\rho < \varepsilon^{-10}$ must hold on $\{\rho_1 \neq \rho_2\}$. In particular, we conclude that there is a subset $K \subseteq Q$ of positive measure on which $0 < \rho(x) < \varepsilon^{-10}$ and $T(x) \neq x$ for all $x \in K$ (fix versions of $T,\rho,K$ so that we can speak of ``all" points instead of ``almost all" points). Note here we have crucially used our original assumption that $\rho_1 \neq \rho_2$. By Lusin's theorem, we may assume after passing to a further subset that $K$ is compact and $T,\rho$ are continuous when restricted to $K$.

	Let $x \in K$ be a Lebesgue point of $K$. Then $T(x) \neq x$, and $K$ contains most of the volume of a small ball around $x$, so there is a Lebesgue point $y$ of $K$ in such a ball so that $|y-T(x)|<|x-T(x)|$. Let $B_x,B_y$ be even smaller disjoint closed balls around $x,y$, such that
	\begin{equation}\label{eq:better}
		|\hat y - \hat z|< |\hat x - \hat z|, \ \ \forall \hat x\in B_x,\ \forall \hat y\in B_y,\ \forall\hat z\in T(K \cap B_x)
	\end{equation}
	(note $T(K \cap B_x)$ is contained in a small neighborhood of $T(x)$ because $T|_K$ is continuous).
	Without loss of generality, we reduce the radii of $B_x$ or $B_y$ further, if necessary, to have $\cL^2(K \cap B_x) = \cL^2(K \cap B_y).$
	Let $\rho'$ be a probability measure obtained from $\rho$ by subtracting a small amount of mass from $K \cap B_x$ and adding back that same amount of mass to $K \cap B_y$. That is, 
	$$\rho' = \rho-\alpha_0 {{\mathbbm{1}}}_{(K \cap B_x)}+ \alpha_0 {{\mathbbm{1}}}_{(K \cap B_y)}$$
	for some $\alpha_0>0$ very small.
	If we take $\alpha_0$ to be sufficiently small, then since $\rho < \varepsilon^{-10}$ on $K$, we will have $\rho' \leq \varepsilon^{-10}$. In addition, since $x,y$ are close, $K \cap B_x$ and $K \cap B_y$ will be in the same rectangle $Q^+$ or $Q^-$, so the total mass of each rectangle will not change, i.e., $\rho'(Q^+) = \rho'(Q^-) = \frac{1}{2}$. Therefore $\rho' \in \mathcal{S}$.
	
	To get a contradiction as desired, it remains to show that $\rho'$ is a better candidate for the projection of $\mu$ to $\mathcal{S}$ than $\rho$. So we want to show $W_p(\rho',\mu) < W_p(\rho,\mu)$. Let $\eta'$ be a transport plan from $\rho'$ to $\mu$ obtained from the optimal plan $\eta$ from $\rho$ to $\mu$ as follows. Take $\eta'$ to agree with $\eta$ except on the mass which was transferred from $K \cap B_x$ to $K \cap B_y$, and send this mass to $T(K \cap B_x)$ in any way which makes $\eta'$ a valid plan from $\rho'$ to $\mu$. More precisely, we can define $\eta'$ as 
	$$
	\eta':=\eta - \alpha_0{{\mathbbm{1}}}_{(K\cap B_x)}\otimes(\pi^y)_\sharp\eta\mres[(K\cap B_x)\times T(K\cap B_x)] +\alpha_0{{\mathbbm{1}}}_{(K\cap B_y)}\otimes(\pi^y)_\sharp\eta\mres [(K\cap B_y)\times T(K\cap B_x)].
	$$
	One readily checks that $\eta'\in\Pi(\rho',\mu)$. Now, we compute
	\begin{align}
		W_p^p(\rho',\mu)&\le\int_{\R^2\times\R^2}|x_1-x_2|^p\d\eta'(x_1,x_2) \notag \\ 
		&= W_p^p(\rho,\mu) - \alpha_0\int_{(K\cap B_x)\times T(K\cap B_x)}|x_1-x_2|^p\d x_1\d\mu(x_2) \notag \\
		& + \alpha_0\int_{(K\cap B_y)\times T(K\cap B_x)}|x_1-x_2|^p\d x_1\d\mu(x_2) \notag \\
		&<W_p^p(\rho,\mu), \label{eqn:better_bd}
	\end{align}
	where the last inequality comes from \eqref{eq:better}. Indeed, for any fixed $\hat x\in B_x\cap K$ and $\hat y\in B_y\cap K$, \eqref{eq:better} gives
	\begin{align} \label{eqn:avg_closer}
		\int_{T(B_x\cap K)}|\hat y - x_2|^p\d\mu(x_2)< \int_{T(B_x\cap K)}|\hat x - x_2|^p\d\mu(x_2).
	\end{align}
	Now \eqref{eqn:better_bd} follows by integrating the left hand side of \eqref{eqn:avg_closer} on $B_y\cap K$ with respect to $\hat y$, integrating the right side on $B_x\cap K$ with respect to $\hat x$, and using Fubini's theorem. This contradicts the optimality of $\rho$.

	\begin{claim} \label{claim:convex}
		The subset $\mathcal{S} \subseteq W_p(\mathbf{R}^2)$ is geodesically convex.
	\end{claim}

	{\it Proof of Claim \ref{claim:convex}.}
	
	Since $\dist(Q^+,Q^-) > \diam(Q^{\pm})$, the optimal transport plan $\eta$ between two measures $\rho,\sigma \in \mathcal{S}$ will never move mass from $Q^+$ to $Q^-$ or vice versa. More rigorously, suppose for a contradiction that $\eta(Q^+ \times Q^-) > 0$. From the definition of $S$,
	\begin{align*}
		\frac{1}{2} &= \rho(Q^+) = \eta(Q^+ \times Q^+) + \eta(Q^+ \times Q^-), \\
		\frac{1}{2} &= \sigma(Q^+) = \eta(Q^+ \times Q^+) + \eta(Q^- \times Q^+),
	\end{align*}
	so
	\begin{align*}
		c := \eta(Q^- \times Q^+) = \eta(Q^+ \times Q^-) > 0.
	\end{align*}
	Indeed, to see that we have equality in the previous line, we note that if $c = \eta(Q^- \times Q^+)$, then one must have that $ \eta(Q^- \times Q^-)=\frac12 - c$. Therefore, $\eta(Q^+ \times Q^+) = \frac12 - c$, and so $\eta(Q^+ \times Q^-) = \frac12 - \left(\frac12-c\right) = c$.
	
	Let $\tilde{\eta}$ be a new transport plan from $\rho$ to $\sigma$ defined by
	\begin{align*}
		\int_{Q \times Q} f \d\tilde{\eta}
		= &\int_{Q^+ \times Q^+} f \d\eta
		+ \int_{Q^- \times Q^-} f \d\eta
		\\&+ \frac{1}{c}\int_{Q^+ \times Q^-} \int_{Q^- \times Q^+} f(x,w) \d\eta(z,w) \d\eta(x,y)
		\\&+ \frac{1}{c}\int_{Q^- \times Q^+} \int_{Q^+ \times Q^-} f(x,w) \d\eta(z,w) \d\eta(x,y)
	\end{align*}
	for all continuous test functions $f \in C(Q \times Q)$. The first two terms here say that $\tilde{\eta}$ agrees with $\eta$ on all the mass which $\eta$ keeps within the same rectangle. The two double integrals say that $\tilde{\eta}$ takes the mass which $\eta$ moves from $Q^{\pm}$ to $Q^{\mp}$ and instead moves it to the mass in $Q^{\pm}$ which, according to $\eta$, comes from $Q^{\mp}$. Thus $\tilde{\eta}$ takes $Q^+$ into $Q^+$ and $Q^-$ into $Q^-$. It's easy to check that the marginals of $\tilde{\eta}$ are $\rho,\sigma$ (this is why the normalization $\frac{1}{c}$ is there), so $\tilde{\eta} \in \Pi(\rho,\sigma)$ is a valid transport plan. Furthermore, $\tilde{\eta}$ is better than $\eta$ because
	\begin{align*}
		\frac{1}{c}\int_{Q^{\pm} \times Q^{\mp}} \int_{Q^{\mp} \times Q^{\pm}} |x-w|^p \d\eta(z,w) \d\eta(x,y)
		&\leq \frac{1}{c}\int_{Q^{\pm} \times Q^{\mp}} \int_{Q^{\mp} \times Q^{\pm}} \diam(Q^{\pm})^p \d\eta(z,w) \d\eta(x,y)
		\\&= c \diam(Q^{\pm})^p
		\\&< c \dist(Q^+,Q^-)^p
		\\&\leq \int_{Q^{\pm} \times Q^{\mp}} |x-w|^p \d\eta.
	\end{align*}
	This contradicts the optimality of $\eta$, proving that indeed $\eta$ cannot move mass from $Q^+$ to $Q^-$ or vice versa. The geodesic connecting $\rho,\sigma$ is $[0,1]\ni t\mapsto \rho_t :=((1-t)x + ty)_{\#}\d\eta(x,y)$, which remains in $\mathcal{S}$, so $\mathcal{S}$ is geodesically convex.

	We have now established the desired properties of $\mathcal{S}$. It remains to construct suitable $\mu,\nu$. For $\varepsilon>0$ small, let
	\begin{align*}
		\mu = \frac{1}{2} \delta_{(-1,\varepsilon)} + \frac{1}{2} \delta_{(1,-\varepsilon)}
		\qquad \text{and} \qquad
		\nu = \frac{1}{2} \delta_{(1,\varepsilon)} + \frac{1}{2} \delta_{(-1,-\varepsilon)}
	\end{align*}
	(these measures are familiar from Lemma \ref{lem:non_busemann} because they are the same as $\mu_1,\nu_1$ in the proof of that lemma). Clearly $W_p(\mu,\nu) \lesssim \varepsilon$. By the saturation argument of \cite[Proposition 5.2]{DePMesSanVel},
	\begin{align*}
		{\rm{P}}_{\mathcal{S}}^{p}[\mu] = \varepsilon^{-10}{\mathbbm{1}}_{(B^+ \cup B^-) \cap Q} \d x,
	\end{align*}
	where $B^{\pm}$ is the ball around $(\mp 1, \pm \varepsilon)$ such that  $B^{\pm} \cap Q$ has area $\frac{1}{2}\varepsilon^{10}$ (note that $B^+ \cap Q$ is disjoint from $B^- \cap Q$). Then $B^{\pm}$ has radius $\leq 2$, so ${\rm{P}}_{\mathcal{S}}^{p}[\mu]$ is supported in $Q_{\mu} = Q_{\mu}^+ \cup Q_{\mu}^-$, where
	\begin{align*}
		Q_{\mu}^+ = [-1,-0.9] \times {[2,2.1]}
		\qquad \text{and} \qquad
		Q_{\mu}^- = [0.9,1] \times {[-2.1,-2]}
	\end{align*}
	(assuming of course that $\varepsilon$ is small enough). Similarly, ${\rm{P}}_{\mathcal{S}}^{p}[\nu]$ is supported in $Q_{\nu} = Q_{\nu}^+ \cup Q_{\nu}^-$, where
	\begin{align*}
		Q_{\nu}^+ = [0.9,1] \times {[2,2.1]}
		\qquad \text{and} \qquad
		Q_{\nu}^- = [-1,-0.9] \times {[-2.1,-2]}.
	\end{align*}
	Thus
	\begin{align*}
		W_p({\rm{P}}^{p}_{\mathcal{S}}[\mu],{\rm{P}}^{p}_{\mathcal{S}}[\nu])
		&\geq \dist(Q_{\mu},Q_{\nu})
		\geq 1.
		\qedhere
	\end{align*}
\end{proof}

We now use the constructions in Lemmas \ref{lem:non_busemann} and \ref{lem:projection_fail} to prove Theorems \ref{thm:Wp_not_busemann} and \ref{thm:Wp_not_nonexpansive}, respectively.

\begin{proof}[Proof of Theorem \ref{thm:Wp_not_busemann}]
	Once again we set $d=2$ for ease of notation, but the same construction works in higher dimensions. It suffices to check that given any open subspace $\mathcal{U} \subseteq W_p(\mathbf{R}^2)$, there are geodesics $(\mu_t)_{t\in[0,1]},(\nu_t)_{t\in[0,1]}$ in $\mathcal{U}$ with $\mu_0 = \nu_0$ and $[0,1]\ni t \mapsto W_p(\mu_t,\nu_t)$ non-convex. Fix $\rho \in \mathcal{U}$. Since $\rho$ has finite $p$th moment, it can be approximated in $W_p$ by a measure with compact support. Thus we may assume from the start that $\rho$ has compact support. Let $h \in \mathbf{R}^2$ far outside the support of $\rho$, and let $(\mu_t')_{t\in[0,1]},(\nu_t')_{t\in[0,1]}$ be the $h$-translates of the geodesics in the proof of Lemma \ref{lem:non_busemann}, so
	\begin{align*}
		\mu_t' = \frac{1}{2} \delta_{h+(-t,1+\varepsilon-t)} + \frac{1}{2} \delta_{h+(t,-1-\varepsilon+t)}
		\qquad &\text{and} \qquad
		\nu_t' = \frac{1}{2} \delta_{h+(t,1+\varepsilon-t)} + \frac{1}{2} \delta_{h+(-t,h-1-\varepsilon+t)}
	\end{align*}
	(fix $\varepsilon$ small here). Then take
	\begin{align*}
		\mu_t = (1-\varepsilon)\rho + \varepsilon\mu_t'
		\qquad \text{and} \qquad
		\nu_t = (1-\varepsilon)\rho + \varepsilon\nu_t'.
	\end{align*}
	These are geodesics for $h$ large enough (so that $\spt(\rho)$ is far away from $\spt(\mu_t')$ and $\spt(\nu_t')$ for all $t\in[0,1]$), and then they lie in $\mathcal{U}$ for $\varepsilon$ small enough depending on $h$. They satisfy $\mu_0 = \nu_0$ and $W_p^p(\mu_1,\nu_1) \lesssim \varepsilon^{1+p}$, but $W_p^p(\mu_{\frac{1}{2}},\nu_{\frac{1}{2}}) \gtrsim \varepsilon$, so $t \mapsto W_p(\mu_t,\nu_t)$ is non-convex.
\end{proof}

\begin{proof}[Proof of Theorem \ref{thm:Wp_not_nonexpansive}]
	As above, we work in $\mathbf{R}^2$ for simplicity, but everything carries over to $\mathbf{R}^d$. It suffices to check that given any open subspace $\mathcal{U} \subseteq W_p(\mathbf{R}^2)$, there is a geodesically convex and $W_p(\mathbf{R}^2)$-Chebyshev subspace $\mathfrak{S} \subseteq \mathcal{U}$ such that the projection of $\mathcal{U}$ onto $\mathfrak{S}$ is not nonexpansive. As in the proof of Theorem \ref{thm:Wp_not_busemann}, let $\rho$ be a compactly supported measure in $\mathcal{U}$. Fix $h \in \mathbf{R}^2$ far outside the support of $\rho$, and then fix $\varepsilon$ small depending $h$.
	Let $\mathcal{S} \subseteq W_p(\mathbf{R}^2)$ be as in Lemma \ref{lem:projection_fail}, and let $\mathcal{S}'$ be the $h$-translate of $\mathcal{S}$, so $\mathcal{S}'$ is the set of measures $\sigma$ supported in $Q = Q^+ \cup Q^-$, where
	\begin{align*}
		Q^+ = h+([-1,1] \times [2,3])
		\qquad \text{and} \qquad
		Q^- = h+([-1,1] \times [-3,-2]),
	\end{align*}
	with density bounded above by $\varepsilon^{-10}$ and with $\sigma(Q^+) = \sigma(Q^-) = \frac{1}{2}$. Then take
	\begin{align*}
		\mathfrak{S} = \{(1-\varepsilon)\rho + \varepsilon\sigma : \sigma \in \mathcal{S}'\}.
	\end{align*}
	Since $h$ is far away from $\spt\rho$, the optimal plan between two measures in $\mathfrak{S}$ is the identity on the $(1-\varepsilon)\rho$ part. It follows from Lemma \ref{lem:projection_fail} that $\mathfrak{S}$ is geodesically convex. In addition, the proof of Claim \ref{claim:proximal} in Lemma \ref{lem:projection_fail} carries over \textit{mutatis mutandis} to show that $\mathfrak{S}$ is Chebyshev.
	For $\varepsilon$ small enough (having already fixed $h$), we have $\mathfrak{S} \subseteq \mathcal{U}$.
	
	Now let $\mu',\nu'$ be the $h$-translates of the measures in Lemma \ref{lem:projection_fail}, so
	\begin{align*}
		\mu' = \frac{1}{2} \delta_{h+(-1,\varepsilon)} + \frac{1}{2} \delta_{h+(1,-\varepsilon)}
		\qquad \text{and} \qquad
		\nu' = \frac{1}{2} \delta_{h+(1,\varepsilon)} + \frac{1}{2} \delta_{h+(-1,-\varepsilon)}.
	\end{align*}
	Set
	\begin{align*}
		\mu = (1-\varepsilon)\rho + \varepsilon\mu'
		\qquad \text{and} \qquad
		\nu = (1-\varepsilon)\rho + \varepsilon\nu'
	\end{align*}
	(again, $\mu,\nu \in \mathcal{U}$ for $\varepsilon$ sufficiently small).
	Using once more that $h$ is far from $\spt\rho$, we have
	\begin{align*}
		{\rm P}^{p}_{\mathfrak{S}}[\mu]
		= (1-\varepsilon)\rho + \varepsilon {\rm P}^{p}_{\mathcal{S}'}[\mu']
		\qquad \text{and} \qquad
		{\rm P}^{p}_{\mathfrak{S}}[\nu]
		= (1-\varepsilon)\rho + \varepsilon {\rm P}^{p}_{\mathcal{S}'}[\nu'].
	\end{align*}
	Now clearly $W_p^p(\mu,\nu) \lesssim \varepsilon^{1+p}$, but by Lemma \ref{lem:projection_fail} combined with the fact that the optimal plan between ${\rm P}^{p}_{\mathfrak{S}}[\mu]$ and ${\rm P}^{p}_{\mathfrak{S}}[\nu]$ is the identity on $(1-\varepsilon)\rho$,
	\begin{align*}
		W_p^p({\rm P}^{p}_{\mathfrak{S}}[\mu],{\rm P}^{p}_{\mathfrak S}[\nu]) \gtrsim \varepsilon.
	\end{align*}
	Thus the projection of $\mathcal{U}$ onto $\mathfrak{S}$ is not nonexpansive.
\end{proof}

\medskip

\noindent
{\sc Acknowledgements.} \, We thank Hugo Lavenant for his feedback on an earlier version of the manuscript.
The first author was partially supported by the National Science Foundation (NSF) Graduate Research Fellowship Program (GRFP) under Grant No. DGE-2039656. The second author was partially supported by the Engineering and Physical Sciences Research Council (EPSRC) under Award No. EP/X020320/1.

\appendix

\section{Some results from optimal transport}\label{sec:prelim}

Some properties of the projection operator ${\rm{P}}_\Om^2$ onto the set $\cK_2(\Om)$ were studied in \cite{DePMesSanVel}. In particular, arguments verbatim to the ones presented there (see in particular \cite[Proposition 5.2]{DePMesSanVel}) yield the following lemma.

\begin{lem} \label{lem:P_defined}
	Let $p\in(1,+\infty)$ and let $\Om\subseteq\R^d$ be closed and convex with $
		\cL^d(\Om)\ge 1$. Then for $\cK=\cK_p(\Om)$ and for any $\mu\in\cP_p(\Om)$, the problem \eqref{def:proj} has a unique solution ${\rm{P}}_\Om^p[\mu]$. Moreover, there exists $B\subseteq\Om$ (depending on $\mu,p,\Omega$) Borel measurable such that
	$${\rm{P}}_\Om^p[\mu]=\mu^{\rm{ac}}{\mathbbm{1}}_B+ \mathbbm{1}_{\Om\setminus B},$$
	where $\mu^{\rm{ac}}$ stands for the absolutely continuous part of $\mu$ with respect to $\cL^{d}$. Here we have used the notation $\mathbbm{1}_{B}$ for the indicator function of $B$.
	Furthermore, for a Borel set $X\subset\R^{d}$, $\mathbbm{1}_{X}$ is interpreted as a Borel measure by $\mathbbm{1}_{X}(A) := \cL^{d}(X\cap A)$, for any Borel set $A$.	
\end{lem}

\begin{rem}
	In the case of $p=2$, the projection ${\rm{P}}_{\Omega}^2$ behaves well with respect to interpolation along {\it generalized geodesics}. Let $\mu,\nu_0,\nu_1 \in \mathcal{P}_2(\Omega)$ with $\mu$ absolutely continuous. Then there are optimal maps $T_0,T_1$ which send $\mu$ to $\nu_0,\nu_1$ respectively. For $t \in [0,1]$, define the generalized geodesic connecting $\nu_0$ and $\nu_1$ with respect to $\mu$ by $\nu_t = (T_t)_{\sharp}\mu$, where $T_t = (1-t)T_0 + tT_1$. In \cite{MauRouSan}, using the displacement 1-convexity of $W_2^2(\mu,\cdot)$ along generalized geodesics, i.e. for all $t\in[0,1]$
	\begin{align*}
		W_2^2(\mu,\nu_t) \leq (1-t)W_2^2(\mu,\nu_0) + tW_2^2(\mu,\nu_1) - t(1-t)W_2^2(\nu_0,\nu_1),
	\end{align*}
	it was shown that ${\rm{P}}_\Om^2$ is locally $\frac12$-H\"older continuous. Since this argument is relying on the ``Hilbertian like'' behavior of $W_2$, it is unclear to us whether such reasoning could be carried through for $p\neq 2$.
\end{rem}

The following result is well-known to experts. However, for the convenience of the reader, we supply a sketch of its proof below.

\begin{lem}\label{lem:admissibles_convex}
	Let $\Om\subseteq\R^d$ be closed and convex such that $\cL^d(\Om)\ge 1$ and $p\in(1,+\infty)$. The subspace $\cK_p(\Omega)$ defined in \eqref{def:Kp} is closed and geodesically convex in $(\cP_p(\Om),W_p)$.
\end{lem}

\begin{proof}
	The closedness of $\cK_p(\Om)$ is straight forward. 
	
	Let $\mu_0,\mu_1\in\cK_p(\Om)$. To show that $\cK_p(\Om)$ is geodesically convex, we will show that the constant speed geodesic $[0,1]\ni t\mapsto\mu_t$ connecting $\mu_0$ to $\mu_1$ is absolutely continuous with a density bounded above by $1$. This is a consequence of \cite[Theorem 7.28, Proposition 7.29]{San} and \cite[Corollary 19.5]{Vil}, but for completeness we provide a sketch of this proof.
	
	First, by \cite[Lemma 3.14]{Kel}, we have that $\mu_t\ll\cL^d\mres\Om$ for all $t\in[0,1]$.
	
	To show the upper bound on $\mu_t$, we rely on the interpolation inequality for the Jacobian determinants of optimal transport maps provided in \cite[Theorem 3.13]{Kel} (see also \cite{CorMcCSch} for $p=2$). Let $\phi:\Om\to\R$ be the unique $c$-concave Kantorovich potential in the optimal transport of $\mu_0$ onto $\mu_1$, where $c(x,y) = \frac{1}{p}|x-y|^p$. Then, by \cite[Theorem 3.4]{Kel}, $T(x):=x-|\nabla\phi(x)|^{q-2}\nabla\phi(x)$ (where $1/p+1/q=1$) is the unique optimal transport map between $\mu_0$ and $\mu_1$. Moreover, $T_t(x):=x-t|\nabla\phi(x)|^{q-2}\nabla\phi(x)$ is the unique optimal transport map between $\mu_0$ and $\mu_t$.
	
	Let us denote by $\Om_{\rm{id}}\subset\Om$ the set where $\phi$ is differentiable and $\nabla\phi=0$. Then, reasoning as in \cite{Kel}, we know that there exists a set $B\subseteq \Om\setminus\Om_{\rm{id}}$ of full measure such that $\phi $ is twice differentiable on $B$ with $\det(DT(x))>0$ if $x\in B$. Then by \cite[Theorem 3.13]{Kel} we have
	\begin{equation}\label{eq:Jac}
		\det(DT_t(x))^{\frac{1}{d}}\ge (1-t) + t \det(DT(x))^{\frac1d}.
	\end{equation}
	We remark that because our underlying space $\Om$ is flat, the volume distortion coefficients present in the previous inequality (stated in \cite{Kel} for general Finslerian manifolds) become 1.
	
	By \eqref{eq:Jac}, if $\det(DT(x))\ge 1$, we conclude that  $\det(DT_t(x))\ge 1$, while if $\det(DT(x))\le 1$, then $\det(DT_t(x))\ge \det(DT(x))$. In conclusion, 
	$$
	\det(DT_t(x))\ge \min\{1,\det(DT(x))\}.
	$$
	
	Now, since $\mu_t=(T_t)_\sharp\mu_0$, when restricted to the set $B$, the change of variable formula yields
	\begin{align*}
		\mu_t(T_t(x)) = \frac{\mu_0(x)}{\det(DT_t(x))}\le \frac{\mu_0(x)}{\min\{1,\det(DT(x))\}}\le\max\{\mu_0(x),\mu_1(T(x))\}\le 1.
	\end{align*}
	When restricted to the relative complement of $B$, $T_t$ essentially is the identity map, where the upper bound is also clearly preserved. The result follows. 
\end{proof}

\end{document}